\documentclass[12pt]{article}
\usepackage[latin1]{inputenc}
\usepackage[british]{babel}
\usepackage{lmodern}
\usepackage[T1]{fontenc}
\usepackage[paper=a4paper, left=28mm, right=25mm, top=29mm, bottom=25mm]{geometry}
\usepackage{latexsym,amsfonts,amsmath,graphics}
\usepackage{epsfig}
\usepackage{enumitem}
\usepackage{amssymb,mathrsfs}
\usepackage{verbatim}
\newtheorem{theorem}{Theorem}
\newtheorem{lemma}{Lemma}
\newtheorem{corollary}{Corollary}

\newtheorem{remark}{Remark}

\newenvironment{proof}{\begin{trivlist} \item[\hskip\labelsep{\it Proof.}]}{$\hfill\Box$\end{trivlist}}

\newcommand{\rd}{\,\mathrm{d}}

\newcommand{\bsx}{\boldsymbol{x}}

\newcommand{\NN}{\mathbb{N}}

\newcommand{\ZZ}{\mathbb{Z}}

\newcommand{\sym}{{\rm sym}}
\newcommand{\cH}{\mathcal{H}}
\newcommand{\cP}{\mathcal{P}}

\newcommand{\vecs}{\boldsymbol{\sigma}}

\newcommand{\qqq}{\mathbb{Q}^{\ast}(2^m)}

\allowdisplaybreaks

\title{An exact formula for the $L_2$ discrepancy of the symmetrized Hammersley point set}
\author{Ralph Kritzinger \thanks{The author is supported by the Austrian Science Fund (FWF): Project F5509-N26, which is a part of the Special Research Program "Quasi-Monte Carlo Methods: Theory and Applications".}}
\date{}

\begin{document}

\maketitle

\begin{abstract}
The process of symmetrization is often used to construct point sets with low $L_p$ discrepancy.
In the current work we apply this method to the shifted Hammersley point set.
It is known that for every shift this symmetrized point set
achieves an $L_p$ discrepancy of order $\mathcal{O}\left(\sqrt{\log{N}}/N\right)$ for $p\in [1,\infty)$, 
which is best possible in the sense of results by Roth, Schmidt and Hal\'{a}sz. In this paper we present an exact formula for
the $L_2$ discrepancy of the symmetrized Hammersley point set, which shows in
particular that it is independent of the choice for the shift.
\end{abstract} 

\centerline{\begin{minipage}[hc]{130mm}{
}
\end{minipage}}

 \allowdisplaybreaks
 
\section{Introduction and statement of the result}

 The local discrepancy $\Delta(\alpha,\beta,\mathcal{P})$ of an $N$-element point set $\mathcal{P}=\{\bsx_0,\dots,\bsx_{N-1}\}$ in the unit square $\left[0,1\right)^2$ is defined as
\[ \Delta(\alpha,\beta,\mathcal{P})=A(\left[0,\alpha\right)\times \left[0,\beta\right),\mathcal{P})-N\alpha\beta \]
for $\alpha, \beta \in \left(0,1\right]$. In this definition $A(\left[0,\alpha \right)\times \left[0,\beta\right),\mathcal{P})$ is the number of indices $0\leq n\leq N-1$ satisfying $\bsx_n \in \left[0,\alpha\right)\times \left[0,\beta\right)$.  The $L_p$ discrepancy of a point set $\mathcal{P}$ in $\left[0,1\right)^2$ is defined as
\[L_p(\mathcal{P})=\frac{1}{N}\left(\int_{0}^{1}\int_{0}^{1} |\Delta(\alpha,\beta,\mathcal{P})|^p \rd\alpha \rd\beta\right)^{\frac{1}{p}} \]
for $p\in[1,\infty)$. For $p\to \infty$ we obtain the notable star discrepancy. In this work we do not study this kind of discrepancy directly, but it should be mentioned that there is a remarkable asymptotic relation between the $L_p$ discrepancy and the star discrepancy (see \cite{Hor}). The $L_p$ discrepancy is a quantitative measure for the irregularity of distribution of a point set $\cP$ in $[0,1)^2$, see e.g. \cite{DT97,kuinie,mat}. It is also related to the worst-case integration error of a quasi-Monte Carlo rule, see e.g. \cite{DP10,LP14,Nied92}. It is well known that for every $p\in [1,\infty)$ there exists a constant $c_p > 0$ with the following
property: for the $L_p$ discrepancy of any point set $\mathcal{P}$ consisting of $N$ points in $[0, 1)^2$ we
have 
\begin{equation} \label{roth} L_p(\mathcal{P}) \geq c_p \frac{\sqrt{\log{N}}}{N}, \end{equation}
where $\log$ denotes the natural logarithm.
This was first shown by Roth \cite{Roth} for $p = 2$ and hence for all $p \in [2,\infty)$ and later by
Schmidt \cite{schX} for all $p\in(1,2)$. The case $p=1$ was verified by Hal\'{a}sz \cite{hala}.

Here we consider digit shifted Hammersley point sets. Let therefore $m$ be a positive integer and $\vecs=(\sigma_1,\sigma_2,\dots,\sigma_m) \in \{0,1\}^m$ a dyadic shift. We define the point set
$$ \cH_m(\vecs):=\left\{\left(\frac{t_m}{2}+\frac{t_{m-1}}{2^2}+\dots+\frac{t_1}{2^m},\frac{s_1}{2}+\frac{s_2}{2^2}+\dots+\frac{s_m}{2^m}\right):
t_1,\dots,t_m \in \{0,1\}\right\}, $$
where $s_j=t_j \oplus \sigma_j$ for all $j\in \{1,\dots,m\}$ (the operation $\oplus$ denotes addition modulo 2). The point set $\cH_m(\vecs)$ contains $2^m$ elements. We obtain the classical Hammersley point set $\cH_m$ with $2^m$ points by choosing $\vecs=(0,0,\dots,0)$. Additionally, we define
the $m$-tuple $\vecs^{\ast}=(\sigma_1^{\ast},\sigma_2^{\ast},\dots,\sigma_m^{\ast})$ by $\sigma_j^{\ast}=\sigma_j \oplus 1$ for all
$j\in \{1,\dots,m\}$. Then we introduce the symmetrized Hammersley point set $\cH_m^{\sym}(\vecs)$ as
$$ \cH_m^{\sym}(\vecs):=\cH_m(\vecs)\cup \cH_m(\vecs^{\ast}). $$
This point set has $2^{m+1}$ elements and can be regarded as symmetrized, since $\cH_m^{\sym}(\vecs)$ may also
be written as the union of $\cH_m(\vecs)$ with the point set
$$ \left\{ \left(x,1-\frac{1}{2^m}-y\right): (x,y)\in \cH_m(\vecs) \right\}. $$
 Figure~\ref{examples} shows examples of two symmetrized Hammersley point sets.

\begin{figure}[ht] \label{examples}
     \centering
     {\includegraphics[width=60mm]{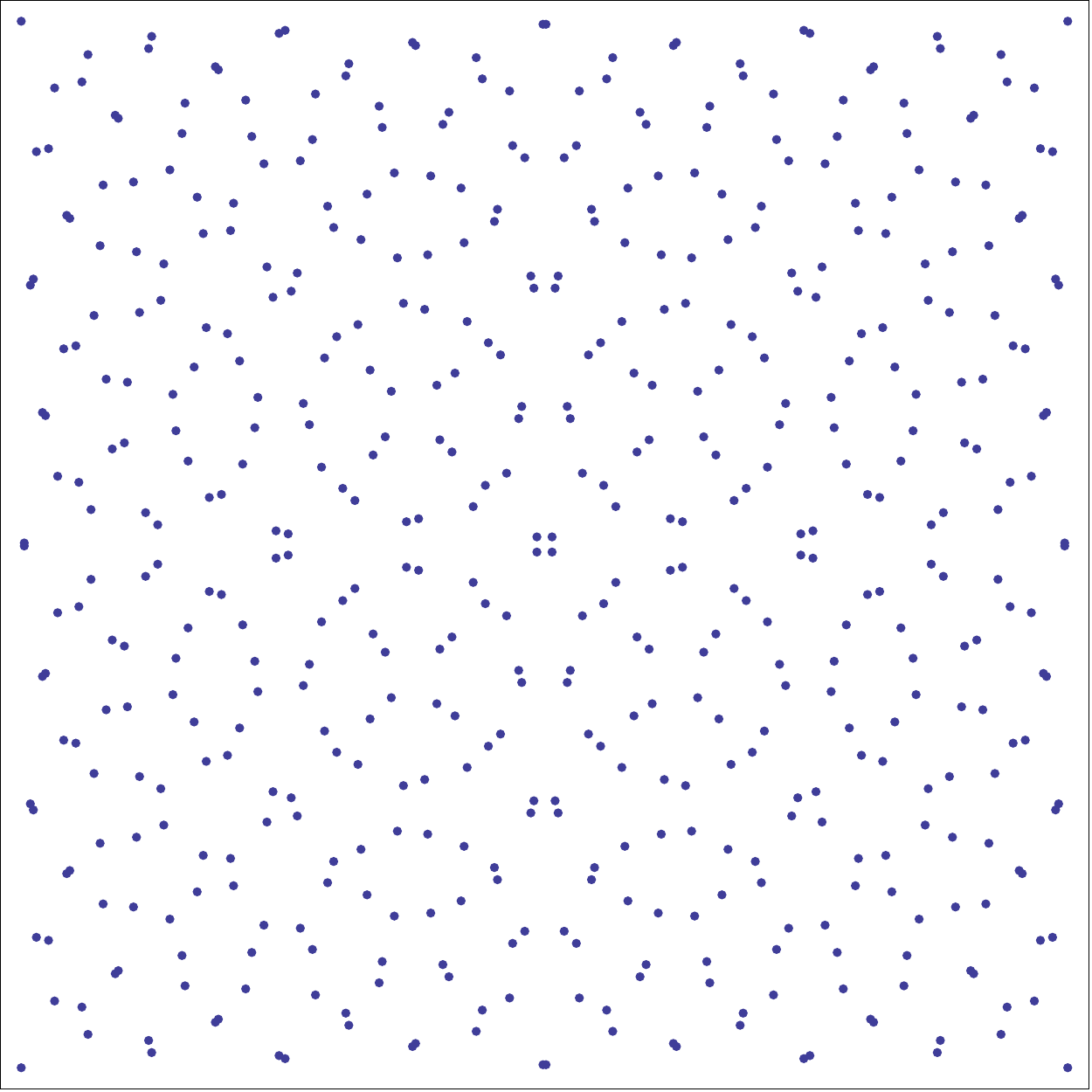}}
     \hspace{.1in}
     {\includegraphics[width=60mm]{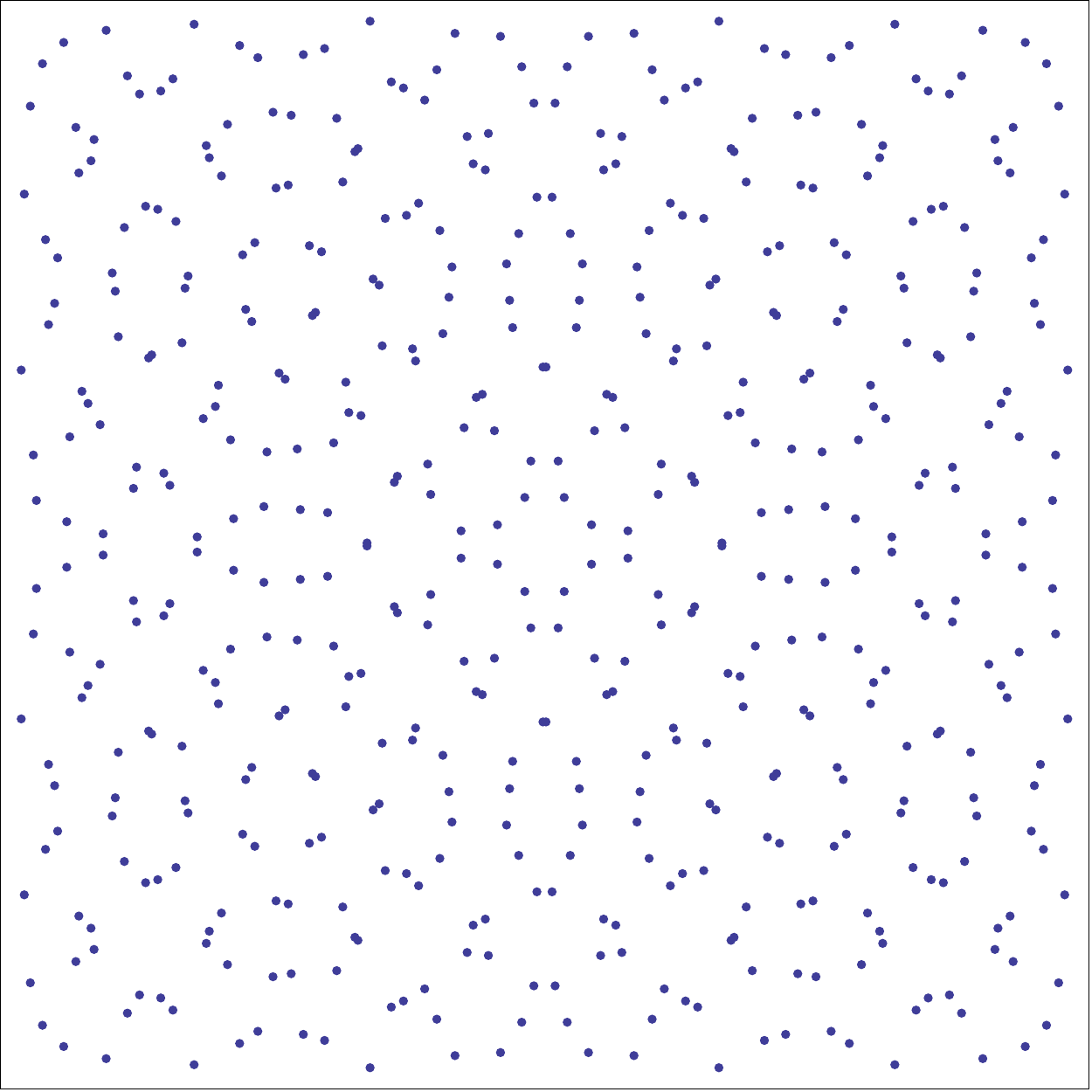}}
     \caption{The symmetrized Hammersley point sets $\cH_8^{\sym}(\vecs_i)$ for $i=1,2$, where $\vecs_1=(0,0,0,0,0,0,0,0)$ and $\vecs_2=(0,1,0,1,0,1,0,1)$. The $L_2$ discrepancy is $0.00255571\dots$ in both cases.}
     \label{f1}
\end{figure}
The concept of symmetrizing point sets plays an important role in finding point sets with the optimal order
of $L_p$ discrepancy in the sense of \eqref{roth}. Davenport \cite{daven} used this method in 1956 to construct for the first
time a two-dimensional point set with an $L_2$ discrepancy of order $\mathcal{O}\left(\sqrt{\log{N}}/N\right)$, and therefore showing that the lower bound \eqref{roth} is sharp for $p=2$. For this reason, the symmetrization method we use here is often referred to as Davenport's reflection principle. 

It is known that $L_p(\cH_m)$ is only of order $\mathcal{O}((\log{N})/N)$ for all $p\in[1,\infty)$ (see \cite{Pill}). However, in \cite[Theorem 2]{HKP14} it was shown with tools from harmonic analysis (the Haar function system and the Littlewood-Paley inequality) that the symmetrized Hammersley point set achieves an $L_p$ discrepancy of order $\mathcal{O}\left(\sqrt{\log{N}}/{N}\right)$
for all $p\in [1,\infty)$ independently of the shift $\vecs$. This order is best possible in the sense of \eqref{roth}. For the case $p=2$, this result follows already from \cite[Theorem 2]{lp} for the slightly different definition of a symmetrization of the classical Hammersley point set $\cH_m$, namely
$$ \widetilde{\cH}_m^{\sym}:=\cH_m \cup \left\{ \left(x,1-y\right): (x,y)\in \cH_m \right\}. $$
The previously mentioned results have the drawback that they do not deliver an exact value for the implied constant of the leading term of the $L_2$ discrepancy. The aim of this paper is to show an exact formula for the $L_2$ discrepancy of $\cH_m^{\sym}(\vecs)$, which gives not only a concrete constant, but also demonstrates that $L_2(\cH_m^{\sym}(\vecs))$ solely depends on the number of elements $N=2^{m+1}$ and not on the shift $\vecs$ whatsoever.
\begin{theorem} \label{Theo} Let $m\in\NN$ and $\vecs \in \{0,1\}^m$. Then we have
    $$ (2^{m+1}L_2(\cH_m^{\sym}(\vecs)))^2=\frac{m}{24}+\frac{11}{8}+\frac{1}{2^{m}}-\frac{1}{9\cdot 2^{2m+1}},  $$
		which can be displayed in terms of the number of elements $N=2^{m+1}$ as
	$$ L_2(\cH_m^{\sym}(\vecs))=\frac{1}{N}\left(\frac{\log{N}}{24\log{2}}+\frac{4}{3}+\frac{2}{N}-\frac{2}{9N^2}\right)^{\frac12}. 
$$ 
\end{theorem}
We derive the following corollary on the point set $\widetilde{\cH}_m^{\sym}(\vecs)$ defined as the union of $\cH_m(\vecs)$
with the point set
$  \left\{ \left(x,1-y\right): (x,y)\in \cH_m(\vecs) \right\}. $
This point set also has $2^{m+1}$ elements, where some points might coincide.
\begin{corollary} Let $m\in\NN$ and $\vecs \in \{0,1\}^m$. Then we have with $N=2^{m+1}$ 
  $$  L_2(\widetilde{\cH}_m^{\sym}(\vecs))=\frac{1}{N}\sqrt{\frac{\log{N}}{24\log{2}}}+\mathcal{O}\left(\frac{1}{N}\right). $$
\end{corollary}
\begin{proof} From \cite[Lemma 4]{HKP14} we have the relation
    $$ \left| L_2(\widetilde{\cH}_m^{\sym}(\vecs))-L_2(\cH_m^{\sym}(\vecs)) \right|\leq \frac{1}{2^{m+1}}=\frac{1}{N}. $$
		Together with Theorem~\ref{Theo} this inequality yields the result.
\end{proof}

The proof of Theorem~\ref{Theo} relies strongly on techniques developed and employed in the papers \cite{Kri2,Kri1,Lar,Pill}. The methods and results of \cite{Kri2}, where the $L_2$ discrepancy of $\cH_m(\vecs)$ was computed exactly,
are particularly important in order to prove the theorem. We comment on those results in Remark~\ref{remrem}, Lemma~\ref{l2shift} and Remark~\ref{optimalshift}. The fact that we can write the symmetrized Hammersley point set as a union of two shifted Hammersley point sets allows us to employ the same techniques in this paper. The reader is invited to compare Theorem~\ref{Theo} to the result of Kritzer and Pillichshammer as stated in Lemma~\ref{l2shift}.

\begin{remark} \rm \label{remrem} Theorem~\ref{Theo} shows that we cannot expect a lower $L_2$ discrepancy by first shifting the classical Hammersley point set and then symmetrizing it. We can therefore simply symmetrize the classical Hammersley point set itself. This is a remarkably easy construction of a point set with very low $L_2$ discrepancy. However, the coefficient of the leading term $\sqrt{\log{N}}/N$ of $L_2(\cH_m^{\sym}(\vecs))$ is $\sqrt{1/(24\log{2})}\approx 0.2451\dots$, which is slightly higher than for the shifted Hammersley point set
$\cH_m(\vecs)$ under the condition that the number of ones and zeros in $\vecs$ is more or less balanced. In this case $\cH_m(\vecs)$ achieves an $L_2$ discrepancy of optimal order of magnitude in $N$ as shown by Kritzer and Pillichshammer in \cite{Kri2, Kri1}. The coefficient of the leading term of $L_2(\cH_m(\vecs))$ is then $\sqrt{5/(192\log{2})}\approx 0.1938\dots$ (see also Lemma~\ref{l2shift} and Remark~\ref{optimalshift}). The smallest known leading constant is achieved for the $L_2$ discrepancy of digit scrambled Hammersley point sets in base $22$ and has the value $\sqrt{278629/(2811072\log{22})}\approx 0.1790\dots$, as shown in \cite{FPPS09}.
\end{remark}
\begin{remark} \rm A further exact formula for the $L_2$ discrepancy of a symmetrized point set was discovered in
  \cite{bil}. There the authors considered symmetrized Fibonacci lattice point sets, whose $L_2$ discrepancy is also of order $\mathcal{O}\left(\sqrt{\log{N}}/N\right)$. The leading term has a complicated form, but numerical results indicate that the $L_2$ discrepancy of these point sets has a constant around $0.176\dots$ This would be slightly better than the result for digit scrambled Hammersley point sets mentioned in Remark~\ref{remrem}. 
\end{remark}

\section{Auxiliary results}

Throughout this paper, we call a real number $\alpha \in \left[0,1\right)$ $m$-bit if it is contained in the set
$\mathbb{Q}(2^m):=\{0,\frac{1}{2^m},\dots,\frac{2^m-1}{2^m}\}$. Hence, $\alpha$ is of the form $\alpha=\frac{\alpha_1}{2}+\dots+\frac{\alpha_m}{2^m}$, where $\alpha_j \in \{0,1\}$ for all $j \in \{1,\dots,m\}$. We also set $\qqq:=\mathbb{Q}(2^m)\setminus \{0\}$. We write $\Delta_1(\alpha,\beta)$ for the local discrepancy of $\cH_m(\vecs)$, $\Delta_2(\alpha,\beta)$ for the local discrepancy of $\cH_m(\vecs^{\ast})$ and $\Delta_{\sym}(\alpha,\beta)$ for the local discrepancy of $\cH_m^{\sym}(\vecs)$. 

The first lemma, which gives an exact formula for the local discrepancy of $\cH_m(\vecs)$, can be derived 
from a result of Larcher and Pillichshammer in \cite{Lar} and was first stated explicitely in \cite[Lemma 1]{Kri1}. The second assertion in this lemma is a consequence of the fact that the components of all elements in $\cH_m(\vecs)$ are $m$-bit as it has already been pointed out in \cite[Remark 3]{Kri1}.
Here and in the following, $\|x\|:=\min_{z \in \ZZ}|x-z|$ denotes the distance to the nearest integer of a real number $x$.
\begin{lemma} \label{AllgemeineFormel} For the local discrepancy $\Delta(\alpha,\beta)$ of $\cH_m(\vecs)$  we have 
  \begin{enumerate}
       \item $ \Delta(\alpha,\beta)=\sum_{u=0}^{m-1} \|2^u \beta\|(-1)^{\sigma_{u+1}} (\alpha_{m-u}\oplus\alpha_{m+1-j(u)})$ for $m$-bit numbers $\alpha=\frac{\alpha_1}{2}+\dots+\frac{\alpha_m}{2^m}$ and $\beta=\frac{\beta_1}{2}+\dots+\frac{\beta_m}{2^m}$ (we set $\alpha_{m+1}=0$), where $j(u)$ for $0\leq u\leq m-1$ is defined as 
\[ j(u)=\begin{cases}
         0 & \text{if } u=0, \\
			   0 & \text{if } \alpha_{m+1-j}=\beta_j \oplus \sigma_j \text{\, for \,} j=1,\dots,u, \\
					 \max\{j \leq u: \alpha_{m+1-j} \neq \beta_j \oplus \sigma_j \} & \text{otherwise.}
        \end{cases}
\]
       \item $\Delta(\alpha,1)=0$ for $m$-bit $\alpha$ and $ \Delta(\alpha,\beta)=\Delta(\alpha(m),\beta(m))+2^m(\alpha(m)\beta(m)-\alpha\beta)$ for arbitrary $\alpha,\beta \in (0,1]$, where $\alpha(m)$ and $\beta(m)$ are the smallest $m$-bit numbers greater than or equal to $\alpha$ or $\beta$, respectively. (For $\alpha, \beta > 1-2^{-m}$ we choose $\alpha(m)=1$ and $\beta(m)=1$, respectively.)
   \end{enumerate} \end{lemma}

  \begin{lemma} \label{symlemma} For all $\alpha,\beta\in (0,1]$ we have $ \Delta_{\sym}(\alpha,\beta)=\Delta_1(\alpha,\beta)+\Delta_2(\alpha,\beta). $
\end{lemma}
\begin{proof}
   We have 
   \begin{align*}
          \Delta_{\sym}(\alpha,\beta)=& A(\left[0,\alpha\right)\times \left[0,\beta\right),\cH_m^{\sym}(\vecs))-2^{m+1}\alpha\beta \\
            =&  A(\left[0,\alpha\right)\times \left[0,\beta\right),\cH_m(\vecs))+ A(\left[0,\alpha\right)\times \left[0,\beta\right),\cH_m(\vecs^{\ast}))
               -2^m\alpha\beta-2^m\alpha\beta \\
           =& \Delta_1(\alpha,\beta)+\Delta_2(\alpha,\beta)
   \end{align*}
   for all $\alpha,\beta\in (0,1]$.
\end{proof}
Throughout the next lemma, we always write $j_1(u)$ if the function $j(u)$ appearing in the first part of Lemma~\ref{AllgemeineFormel} refers to $\Delta_1(\alpha,\beta)$	and $j_2(u)$ if it refers to $\Delta_2(\alpha,\beta)$.   
  \begin{lemma} \label{Alphalemma}
	Let $\alpha=\frac{\alpha_1}{2}+\dots+\frac{\alpha_m}{2^m}$ and $\beta=\frac{\beta_1}{2}+\dots+\frac{\beta_m}{2^m}$ be $m$-bit.
	\begin{enumerate}
\item For $u_1,u_2 \in \{0,\dots,m-1\}$ with $u_1\neq u_2$ we have
$$ \sum_{\alpha\in\qqq}(\alpha_{m-u_1}\oplus \alpha_{m+1-j_1(u_1)})(\alpha_{m-u_2}\oplus \alpha_{m+1-j_2(u_2)})=2^{m-2}. $$
 \item For $u \in \{0,\dots,m-1\}$ we have
\begin{align*} \sum_{\alpha\in\qqq}&(\alpha_{m-u}\oplus \alpha_{m+1-j_1(u)})(\alpha_{m-u}\oplus \alpha_{m+1-j_2(u)}) \\
   =&  \begin{cases}   2^{m-u-1} & \mbox{if } u\in \{0,1\},  \\
         2^{m-u-1}\left(1+\sum_{j=1}^{u-1}2^j((\gamma_j \oplus 1)\gamma_u+\gamma_j (\gamma_u \oplus 1))\right)  & \mbox{if } u\in\{2,\dots,m-1\}. 
         \end{cases} \end{align*}
In the last expression, we define $\gamma_j:=\beta_j\oplus \sigma_j$ for all $j\in\{1,\dots,m-1\}$.\end{enumerate}
	
\end{lemma}
\begin{proof}  
We mention that Pillichshammer showed in \cite[Lemma 2]{Pill} the formula
$$ \sum_{\alpha\in\qqq}\prod_{i=1}^{k}\left(\alpha_{m-u_i}\oplus \alpha_{m+1-j(u_i)}\right)=2^{m-k} $$
for an integer $1\leq k \leq m-1$ and numbers $u_1,\dots,u_k\in \{0,\dots,m-1\}$ with $u_i\neq u_j$ for $1\leq i\neq j \leq k$, where $j(u)$ refers to the local discrepancy of the classical Hammersley point set.
By studying his proof, one sees that the argumentation does not change at all if we replace some of the $j(u_i)$
appearing in the formula by $j_1(u_i)$ and the others by $j_2(u_i)$, and thus we obtain the claimed identity stated in the first item of this lemma by choosing $k=2$ and replacing $j(u_1)$ by $j_1(u_1)$ and $j(u_2)$ by $j_2(u_2)$. 

We show the second item. 
For $u=0$ we have $j_1(u)=0$ and $j_2(u)=0$ by definition and hence
\begin{align*}
    \sum_{\alpha\in\qqq}(\alpha_m\oplus\alpha_{m+1})(\alpha_m\oplus\alpha_{m+1})=\sum_{\alpha_1,\dots,\alpha_m=0}^{1}\alpha_m=		\sum_{\alpha_1,\dots,\alpha_{m-1}=0}^{1}1=2^{m-1}=2^{m-u-1}.
\end{align*}
If $u=1$, we use the fact that $j_1(1)$ and $j_2(1)$ only depend on $\alpha_m$ and write
\begin{align*} 
 \sum_{\alpha\in\qqq}&(\alpha_{m-1}\oplus\alpha_{m+1-j_1(1)})(\alpha_{m-1}\oplus\alpha_{m+1-j_2(1)}) \\
   =& \sum_{\alpha_m=0}^{1}\left(\sum_{\alpha_1,\dots,\alpha_{m-1}=0}^{1}(\alpha_{m-1}\oplus\alpha_{m+1-j_1(1)})(\alpha_{m-1}\oplus\alpha_{m+1-j_2(1)})\right) \\
	 =& 2^{m-2}\sum_{\alpha_m=0}^{1}\left(\alpha_{m+1-j_1(1)}\alpha_{m+1-j_2(1)}+(\alpha_{m+1-j_1(1)}\oplus 1)(\alpha_{m+1-j_2(1)}\oplus 1)\right).
\end{align*}
We have to distinguish between the cases $\alpha_m=\gamma_1$ and $\alpha_m=\gamma_1\oplus 1$. In the first case we obviously have
$j_1(1)=0$ and $j_2(1)=1$ whereas in the second case we have $j_1(1)=1$ and $j_2(1)=0$. We conclude
\begin{align*} 
 2^{m-2}&\sum_{\alpha_m=0}^{1}\left(\alpha_{m+1-j_1(1)}\alpha_{m+1-j_2(1)}+(\alpha_{m+1-j_1(1)}\oplus 1)(\alpha_{m+1-j_2(1)}\oplus 1)\right) \\
   =& 2^{m-2}\sum_{\alpha_m=\gamma_1}(\alpha_m\oplus 1)+2^{m-2}\sum_{\alpha_m=\gamma_1\oplus 1}(\alpha_m\oplus 1) \\
	 =& 2^{m-2}(\gamma_1\oplus 1)+2^{m-2}\gamma_1=2^{m-2}=2^{m-u-1}.
\end{align*}
We turn to the case $u \in \{2,\dots,m-1\}$. Since $j_1(u)$ and $j_2(u)$ only depend on $\alpha_{m+1-u},\dots,\alpha_m$
but not on $\alpha_1,\dots,\alpha_{m-u}$, we observe that
\begin{align*}
    &\sum_{\alpha\in\qqq} (\alpha_{m-u}\oplus \alpha_{m+1-j_1(u)})(\alpha_{m-u}\oplus \alpha_{m+1-j_2(u)}) \\
		&= \sum_{\alpha_1,\dots,\alpha_m=0}^{1}(\alpha_{m-u}\oplus \alpha_{m+1-j_1(u)})(\alpha_{m-u}\oplus \alpha_{m+1-j_2(u)}) \\
    =&\sum_{\alpha_{m+1-u},\dots,\alpha_{m}=0}^{1}\left(\sum_{\alpha_1,\dots,\alpha_{m-u}=0}^{1}(\alpha_{m-u}\oplus \alpha_{m+1-j_1(u)})(\alpha_{m-u}\oplus \alpha_{m+1-j_2(u)})\right) \\
    =& 2^{m-u-1}\sum_{\alpha_{m+1-u},\dots,\alpha_{m}=0}^{1}\left(\alpha_{m+1-j_1(u)}\alpha_{m+1-j_2(u)}+(\alpha_{m+1-j_1(u)}\oplus 1)(\alpha_{m+1-j_2(u)}\oplus 1)\right) \\
    =& 2^{m-u-1}\sum_{j_1=0}^{u-1}\sum_{\substack{\alpha_{m+1-u},\dots,\alpha_{m}=0 \\ j_1(u)=j_1}}^{1}\left(\alpha_{m+1-j_1}\alpha_{m+1-j_2(u)}+(\alpha_{m+1-j_1}\oplus 1)(\alpha_{m+1-j_2(u)}\oplus 1)\right)  \\
    &+ 2^{m-u-1}\sum_{j_2=0}^{u-1}\sum_{\substack{\alpha_{m+1-u},\dots,\alpha_{m}=0 \\ j_2(u)=j_2}}^{1}\left(\alpha_{m+1-j_1(u)}\alpha_{m+1-j_2}+(\alpha_{m+1-j_1(u)}\oplus 1)(\alpha_{m+1-j_2}\oplus 1)\right) \\
    =:&\, T_1+T_2.
\end{align*}
One might wonder why the sums over $j_1$ and $j_2$ end in $u-1$ instead of $u$ and why they do not coincide. The reason is that $j_1(u)\in \{0,\dots,u-1\}$ implies $j_2(u)=u$ and $j_2(u)\in \{0,\dots,u-1\}$ implies $j_1(u)=u$. This can be seen as follows: $j_1(u)\in \{0,\dots,u-1\}$ implies $a_{m+1-u}=\gamma_u$, because otherwise we would have $j_1(u)=u$. But from the fact that $a_{m+1-u}=\gamma_u\neq \gamma_u\oplus 1$, we immediately derive $j_2(u)=u$. The other way round can be explained analogously. This means that the case $j_2(u)=u$ is actually contained in the sum over $j_1$ and reversely. We find
\begin{align*}
   T_1=& 2^{m-u-1}\sum_{\substack{\alpha_{m+1-u}=\gamma_u \\ \vdots  \\ \alpha_{m-1}=\gamma_{2} \\ \alpha_{m}=\gamma_{1}}}\left(\alpha_{m+1}\alpha_{m+1-j_2(u)}+(\alpha_{m+1}\oplus 1)(\alpha_{m+1-j_2(u)}\oplus 1)\right) \\
     &+ 2^{m-u-1}\sum_{j_1=1}^{u-1}\sum_{\alpha_{m-j_1+2},\dots,\alpha_{m}=0}^{1} \\ &\sum_{\substack{\alpha_{m+1-u}=\gamma_u \\ \vdots \\ \alpha_{m-j_1}=\gamma_{j_1+1} \\ \alpha_{m+1-j_1}=\gamma_{j_1}\oplus1}}\left(\alpha_{m+1-j_1}\alpha_{m+1-j_2(u)}+(\alpha_{m+1-j_1}\oplus 1)(\alpha_{m+1-j_2(u)}\oplus 1)\right) \\
    =& 2^{m-u-1}\left(\gamma_u \oplus 1\right) 
     + 2^{m-u-1}\sum_{j_1=1}^{u-1}2^{j_1-1}\left((\gamma_{j_1}\oplus1)\gamma_u+\gamma_{j_1}(\gamma_u\oplus 1)\right).
\end{align*}

Similarly we argue that
$$ T_2=2^{m-u-1}\gamma_u 
     + 2^{m-u-1}\sum_{j_2=1}^{u-1}2^{j_2-1}\left((\gamma_{j_2}\oplus1)\gamma_u+\gamma_{j_2}(\gamma_u\oplus 1)\right). $$
 Adding $T_1$ and $T_2$ completes the proof of the second item of this lemma.
\end{proof}

\begin{lemma} \label{Betalemma} Let $\beta$ be $m$-bit.
\begin{enumerate}
\item For $u_1,u_2 \in \{0,\dots,m-1\}$ with $u_1\neq u_2$ we have
$$ \sum_{\beta\in\qqq}\|2^{u_1}\beta\|\|2^{u_2}\beta\|=\frac{2^m}{2^4}. $$
\item For $u \in \{0,\dots,m-1\}$ we have
$$ \sum_{\beta\in\qqq}\|2^{u}\beta\|^2=\frac{2^{2m}+2^{2u+1}}{3\cdot 2^{m+2}}. $$
\end{enumerate}
\end{lemma}
\begin{proof} The first formula follows from \cite[Lemma 3 a)]{Pill} and the second one
is \cite[Lemma 3 b)]{Pill}.
\end{proof}

We introduce the parameter $l=l(\vecs):=|\{i\in\{1,\dots,m\}:\sigma_i=0\}|$, i. e. $l$ is the number of components of $\vecs$ which are equal to zero. We use this notation for the rest of this paper.

\begin{lemma} \label{gemischt}
   We have
   $$ \frac{1}{2^{2m}}\sum_{\alpha,\beta\in\qqq}\Delta_1(\alpha,\beta)\Delta_2(\alpha,\beta)=-\frac{m^2}{64}-\frac{l^2}{16}+\frac{lm}{16}-\frac{m}{192}-\frac{5}{144}-\frac{1}{9\cdot 2^{2m+2}}. $$
\end{lemma}

\begin{proof}
In this proof we write for the sake of simplicity
$ A(\alpha,\beta,\vecs,u):=\alpha_{m-u}\oplus \alpha_{m+1-j(u)},$
where we emphasize the dependence of $j(u)$ on $\alpha$, $\beta$ and $\vecs$.
With the first point of Lemma~\ref{AllgemeineFormel} we get
\begin{align*}
   \frac{1}{2^{2m}}&\sum_{\alpha,\beta\in\qqq}\Delta_1(\alpha,\beta)\Delta_2(\alpha,\beta) \\
     =&
    \frac{1}{2^{2m}}\sum_{\alpha,\beta\in\qqq}\left(\sum_{u_1=0}^{m-1}\|2^{u_1}\beta\|(-1)^{\sigma_{u_1+1}}A(\alpha,\beta,\vecs,u_1)\right) \\ &\hspace{100pt}\times
    \left(\sum_{u_2=0}^{m-1}\|2^{u_2}\beta\|(-1)^{\sigma_{u_2+1}^{\ast}}A(\alpha,\beta,\vecs^{\ast},u_2)\right) \\
    =&
    -\frac{1}{2^{2m}}\sum_{\alpha,\beta\in\qqq}\left(\sum_{u_1=0}^{m-1}\|2^{u_1}\beta\|(-1)^{\sigma_{u_1+1}}A(\alpha,\beta,\vecs,u_1)\right) \\ &\hspace{100pt}\times
    \left(\sum_{u_2=0}^{m-1}\|2^{u_2}\beta\|(-1)^{\sigma_{u_2+1}}A(\alpha,\beta,\vecs^{\ast},u_2)\right) \\
    =& -\frac{1}{2^{2m}}\sum_{\alpha,\beta\in\qqq}\left(\sum_{\substack{u_1=0 \\ \sigma_{u_1+1}=0}}^{m-1}\|2^{u_1}\beta\|A(\alpha,\beta,\vecs,u_1)\right)
    \left(\sum_{\substack{u_2=0 \\ \sigma_{u_2+1}=0}}^{m-1}\|2^{u_2}\beta\|A(\alpha,\beta,\vecs^{\ast},u_2)\right) \\
    &+\frac{1}{2^{2m}}\sum_{\alpha,\beta\in\qqq}\left(\sum_{\substack{u_1=0 \\ \sigma_{u_1+1}=0}}^{m-1}\|2^{u_1}\beta\|A(\alpha,\beta,\vecs,u_1)\right)
    \left(\sum_{\substack{u_2=0 \\ \sigma_{u_2+1}=1}}^{m-1}\|2^{u_2}\beta\|A(\alpha,\beta,\vecs^{\ast},u_2)\right) \\
    &+\frac{1}{2^{2m}}\sum_{\alpha,\beta\in\qqq}\left(\sum_{\substack{u_1=0 \\ \sigma_{u_1+1}=1}}^{m-1}\|2^{u_1}\beta\|A(\alpha,\beta,\vecs,u_1)\right)
    \left(\sum_{\substack{u_2=0 \\ \sigma_{u_2+1}=0}}^{m-1}\|2^{u_2}\beta\|A(\alpha,\beta,\vecs^{\ast},u_2)\right) \\
    &-\frac{1}{2^{2m}}\sum_{\alpha,\beta\in\qqq}\left(\sum_{\substack{u_1=0 \\ \sigma_{u_1+1}=1}}^{m-1}\|2^{u_1}\beta\|A(\alpha,\beta,\vecs,u_1)\right)
    \left(\sum_{\substack{u_2=0 \\ \sigma_{u_2+1}=1}}^{m-1}\|2^{u_2}\beta\|A(\alpha,\beta,\vecs^{\ast},u_2)\right) \\
    =:&-R_1+R_2+R_3-R_4.
\end{align*}
With the first part of Lemma~\ref{Alphalemma} and Lemma~\ref{Betalemma} we obtain
\begin{align*} R_2=&\frac{1}{2^{2m}}\sum_{\substack{u_1=0 \\ \sigma_{u_1+1}=0}}^{m-1}\sum_{\substack{u_2=0 \\ \sigma_{u_2+1}=1}}^{m-1}\sum_{\beta\in\qqq}\|2^{u_1}\beta\|\|2^{u_2}\beta\|
\sum_{\alpha\in\qqq}A(\alpha,\beta,\vecs,u_1)A(\alpha,\beta,\vecs^{\ast},u_2) \\
=& \frac{1}{2^{2m}}\sum_{\substack{u_1=0 \\ \sigma_{u_1+1}=0}}^{m-1}\sum_{\substack{u_2=0 \\ \sigma_{u_2+1}=1}}^{m-1}\frac{2^m}{2^4}2^{m-2}
  =\frac{1}{64}l(m-l).
\end{align*}
In the same way we show $ R_3=\frac{1}{64}l(m-l)$. To calculate $R_1$ and $R_4$, we need to distinguish between
the cases where $u_1=u_2$ and where $u_1 \neq u_2$. This leads to
\begin{align*} R_1=&\frac{1}{2^{2m}}\underbrace{\sum_{\substack{u_1=0 \\ \sigma_{u_1+1}=0}}^{m-1}\sum_{\substack{u_2=0 \\ \sigma_{u_2+1}=0}}^{m-1}}_{u_1 \neq u_2} \frac{2^m}{2^4}2^{m-2} \\ &+\frac{1}{2^{2m}}\sum_{\substack{u=0 \\ \sigma_{u+1}=0}}^{m-1}\sum_{\beta\in\qqq}\|2^{u}\beta\|^2\sum_{\alpha\in\qqq}A(\alpha,\beta,\vecs,u)A(\alpha,\beta,\vecs^{\ast},u) \\
=& \frac{1}{64}l(l-1)+\frac{1}{2^{2m}}\sum_{\substack{u=0 \\ \sigma_{u+1}=0}}^{m-1}\sum_{\beta\in\qqq}\|2^{u}\beta\|^2
\sum_{\alpha\in\qqq}A(\alpha,\beta,\vecs,u)A(\alpha,\beta,\vecs^{\ast},u).
\end{align*}
Similarly, we obtain
$$ R_4=\frac{1}{64}(m-l)(m-l-1)+\frac{1}{2^{2m}}\sum_{\substack{u=0 \\ \sigma_{u+1}=1}}^{m-1}\sum_{\beta\in\qqq}\|2^{u}\beta\|^2
\sum_{\alpha\in\qqq}A(\alpha,\beta,\vecs,u)A(\alpha,\beta,\vecs^{\ast},u). $$
Adding $R_1$ to $R_4$ yields
\begin{align*}\frac{1}{2^{2m}}&\sum_{\alpha,\beta\in\qqq}\Delta_1(\alpha,\beta)\Delta_2(\alpha,\beta) \\
     =& -\frac{1}{64}(m^2+4l^2-4lm-m) \\&-\frac{1}{2^{2m}}\sum_{u=0}^{m-1}\sum_{\beta\in\qqq}\|2^{u}\beta\|^2
		  \sum_{\alpha\in\qqq}A(\alpha,\beta,\vecs,u)A(\alpha,\beta,\vecs^{\ast},u).
\end{align*}

Hence, our final task is to compute the last expression in the above line. 
We employ the second part of Lemma~\ref{Alphalemma} and Lemma~\ref{Betalemma} to obtain
\begin{align*}
     \frac{1}{2^{2m}}&\sum_{u=0}^{m-1}\sum_{\beta\in\qqq}\|2^{u}\beta\|^2\sum_{\alpha\in\qqq}A(\alpha,\beta,\vecs,u)A(\alpha,\beta,\vecs^{\ast},u)\\
		=&\frac{1}{2^{2m}}\sum_{u=0}^{m-1}2^{m-u-1}\sum_{\beta\in\qqq}\|2^u\beta\|^2\\
		     &+\frac{1}{2^{2m}}\sum_{u=2}^{m-1}2^{m-u-1}\sum_{\beta\in\qqq}\|2^u\beta\|^2  \sum_{j=1}^{u-1}2^j((\gamma_j \oplus 1)\gamma_u+\gamma_j (\gamma_u \oplus 1)) \\
				=&\frac{1}{2^{m+1}}\sum_{u=0}^{m-1}2^{-u}\frac{2^{2m}+2^{2u+1}}{3\cdot 2^{m+2}} \\
				&+ \frac{1}{2^{m+1}}\sum_{u=2}^{m-1}2^{-u}\sum_{j=1}^{u-1}2^j\sum_{\beta\in\qqq}\|2^u\beta\|^2
				        (\beta_j \oplus \sigma_j \oplus 1)(\beta_u \oplus \sigma_u) \\
				&+ \frac{1}{2^{m+1}}\sum_{u=2}^{m-1}2^{-u}\sum_{j=1}^{u-1}2^j\sum_{\beta\in\qqq}\|2^u\beta\|^2
				        (\beta_j \oplus \sigma_j)(\beta_u \oplus \sigma_u \oplus 1)=:\Sigma_1+\Sigma_2+\Sigma_3. \\
\end{align*}
Finding the value of $\Sigma_1$ is a matter of straightforward calculation. We have 
$$\Sigma_1=\frac{1}{12}\left(1-\frac{1}{2^{2m}}\right).$$
For $\Sigma_2$ we find
\begin{align*}
    \Sigma_2=&\frac{1}{2^{m+1}}\sum_{u=2}^{m-1}2^{-u}\sum_{j=1}^{u-1}2^j\sum_{\substack{\beta_1,\dots,\beta_{j-1},\beta_{j+1},\dots,\beta_{u-1}=0 \\ \beta_j=\sigma_j\\ \beta_u=\sigma_u \oplus 1}}^{1}\sum_{\beta_{u+1},\dots,\beta_m=0}^{1}\|2^u\beta\|^2.
\end{align*}
We remark at this point that $\|2^u\beta\|^2$ only depends on $\beta_{u+1},\dots,\beta_m$. Hence,
$$ \sum_{\beta_{u+1},\dots,\beta_m=0}^{1}\|2^u\beta\|^2=2^{-u}\sum_{\beta_{1},\dots,\beta_m=0}^{1}\|2^u\beta\|^2
= 2^{-u}\sum_{\beta\in\qqq}\|2^u\beta\|^2=2^{-u}\frac{2^{2m}+2^{2u+1}}{3\cdot 2^{m+2}}.$$
We arrive at
\begin{align*}
    \Sigma_2=&\frac{1}{2^{m+1}}\sum_{u=2}^{m-1}2^{-u}\sum_{j=1}^{u-1}2^j2^{u-2}2^{-u}\frac{2^{2m}+2^{2u+1}}{3\cdot 2^{m+2}} \\
		        =&\frac{1}{2^{m+1}}\sum_{u=2}^{m-1}2^{-u}(2^u-2)2^{u-2}2^{-u}\frac{2^{2m}+2^{2u+1}}{3\cdot 2^{m+2}} \\
						=& \frac{m}{96}-\frac{7}{288}+\frac{1}{9\cdot 2^{2m+1}}. \end{align*}
It is clear that $\Sigma_3=\Sigma_2$. Thus, after adding all the results the proof of the lemma is finally complete.
\end{proof}

For the proof of Theorem~\ref{Theo}, we will also need an exact formula for the $L_2$ discrepancy of $\cH_m(\vecs)$. Such a formula was presented in \cite[Theorem 1]{Kri2}.

\begin{lemma}[Kritzer and Pillichshammer] \label{l2shift} Let $m\in\NN$ and $\vecs \in \{0,1\}^m$. We have
   $$ (2^mL_2(\cH_m(\vecs)))^2=\frac{m^2}{64}-\frac{19m}{192}-\frac{lm}{16}+\frac{l^2}{16}+\frac{l}{4}+\frac38+\frac{m}{16\cdot 2^m}-\frac{l}{8\cdot 2^m}+\frac{1}{4\cdot 2^m}-\frac{1}{72\cdot 4^m}.  $$
\end{lemma}

\begin{remark} \label{optimalshift} \rm It follows from Lemma~\ref{l2shift} that the optimal choice for $l$ is $\left\lceil\frac{m-5}{2}+\frac{1}{2^m}\right\rceil$, which leads to 
$$(2^mL_2(\cH_m(\vecs)))^2=\frac{5m}{192}+\mathcal{O}(1)  $$
(see also \cite[Corollary 1]{Kri2}).
This means that we achieve the optimal order of $L_2$ discrepancy for $\cH_m(\vecs)$ in this case. In \cite[Theorem 1]{HKP14} it was shown that
we achieve the optimal order of $L_p$ discrepancy for all $p\in[1,\infty)$ if and only if $|2l-m|=\mathcal{O}(\sqrt{m})$.
\end{remark}

\section{Proof of Theorem~\ref{Theo}}

We apply Lemma~\ref{symlemma} to write
\begin{align} (2^{m+1}L_2(\cH_m^{\sym}(\vecs)))^2=& \int_{0}^{1} \int_{0}^{1}  (\Delta_{\sym}(\alpha,\beta))^2 \rd\alpha \rd\beta \nonumber \\ \nonumber
 =&\int_{0}^{1} \int_{0}^{1}  (\Delta_1(\alpha,\beta))^2 \rd\alpha \rd\beta+\int_{0}^{1} \int_{0}^{1}  (\Delta_2(\alpha,\beta))^2 \rd\alpha \rd\beta\\ \nonumber
 &+2\int_{0}^{1} \int_{0}^{1}  \Delta_1(\alpha,\beta)\Delta_2(\alpha,\beta) \rd\alpha \rd\beta\\ \nonumber
 =&(2^{m}L_2(\cH_m(\vecs)))^2+(2^{m}L_2(\cH_m(\vecs^{\ast})))^2 \\ \label{letztesintegral}
&+2\int_{0}^{1} \int_{0}^{1}  \Delta_1(\alpha,\beta)\Delta_2(\alpha,\beta) \rd\alpha \rd\beta.
  \end{align}
  
We know the values of $(2^mL_2(\cH_m(\vecs)))^2$ and $(2^mL_2(\cH_m(\vecs^{\ast})))^2$ already from Lemma~\ref{l2shift} (where in the latter
case we have to insert $m-l$ instead of $l$ in this formula). This yields
\begin{align*}
  (2^mL_2(\cH_m(\vecs)))^2+(2^mL_2(\cH_m(\vecs^{\ast})))^2=\frac{m^2}{32}+\frac{l^2}{8}-\frac{lm}{8}+\frac{5m}{96}+\frac34+\frac{1}{2^{m+1}}-\frac{1}{9\cdot 2^{2m+2}}.
\end{align*} \\
We split the integrals in \eqref{letztesintegral} in four parts:
 \begin{align*}  \int_{0}^{1} \int_{0}^{1}  \Delta_1(\alpha,\beta)\Delta_2(\alpha,\beta) \rd\alpha \rd\beta=&
 \int_{0}^{1-2^{-m}} \int_{0}^{1-2^{-m}} \Delta_1(\alpha,\beta)\Delta_2(\alpha,\beta) \rd\alpha \rd\beta \\&+\int_{0}^{1-2^{-m}} \int_{1-2^{-m}}^{1} \Delta_1(\alpha,\beta)\Delta_2(\alpha,\beta) \rd\alpha\rd\beta  \\
  &+ \int_{1-2^{-m}}^{1} \int_{0}^{1-2^{-m}} \Delta_1(\alpha,\beta)\Delta_2(\alpha,\beta) \rd\alpha\rd\beta \\ & +\int_{1-2^{-m}}^{1} \int_{1-2^{-m}}^{1} \Delta_1(\alpha,\beta)\Delta_2(\alpha,\beta) \rd\alpha \rd\beta \\ =:&I_1+I_2+I_3+I_4.       \end{align*}
  We can calculate $I_2$, $I_3$ and $I_4$ with aid of the second part of Lemma~\ref{AllgemeineFormel}. Since this proceeds analogously as in the
  proof of \cite[Theorem 1]{Kri1}, we only give the results. We have
  $$ I_2=I_3=\frac{25}{36\cdot 2^m}-\frac{5}{9\cdot 4^m}-\frac{25}{36\cdot 4^m}+\frac{2}{3\cdot 8^m}-\frac{1}{9\cdot 16^m} $$ and
  $$ I_4=\frac{7}{6\cdot 4^m}+\frac{1}{9\cdot 16^m}-\frac{2}{3\cdot 8^m}. $$
  It remains to evaluate $I_1$. We use the second part of Lemma~\ref{AllgemeineFormel} to obtain
  \begin{align*}
      I_1=&  \int_{0}^{1-2^{-m}} \int_{0}^{1-2^{-m}}(\Delta_1(\alpha(m),\beta(m))+2^m(\alpha(m)\beta(m)-\alpha\beta))  \\
         &\hspace{65 pt} \times (\Delta_2(\alpha(m),\beta(m))+2^m(\alpha(m)\beta(m)-\alpha\beta)) \rd \alpha \rd \beta \\
         =&\int_{0}^{1-2^{-m}} \int_{0}^{1-2^{-m}}\Delta_1(\alpha(m),\beta(m))\Delta_2(\alpha(m),\beta(m))\rd \alpha \rd \beta \\
         &+2^m\int_{0}^{1-2^{-m}} \int_{0}^{1-2^{-m}}\Delta_1(\alpha(m),\beta(m))(\alpha(m)\beta(m)-\alpha\beta)\rd \alpha \rd \beta \\
         &+2^m\int_{0}^{1-2^{-m}} \int_{0}^{1-2^{-m}}\Delta_2(\alpha(m),\beta(m))(\alpha(m)\beta(m)-\alpha\beta)\rd \alpha \rd \beta \\
         &+2^{2m}\int_{0}^{1-2^{-m}} \int_{0}^{1-2^{-m}}(\alpha(m)\beta(m)-\alpha\beta)^2\rd \alpha \rd \beta=S_1+S_2+S_3+S_4. \end{align*}
      The value of $S_4$ can be calculated in a straightforward way and is
      $$ S_4=-\frac{1}{72\cdot 16^m}(2^m-1)^2(32\cdot 2^m-25\cdot4^m-8). $$
      The expression $S_2$ was computed in the proof of \cite[Theorem 1]{Kri2} and is given by
      $$ S_2=2^{m-1}\frac{2^{m+1}-1}{4^m}\left(\frac{l(\vecs)}{8}-\frac{m}{16}\right). $$
       Analogously, we have
      $$ S_3=2^{m-1}\frac{2^{m+1}-1}{4^m}\left(\frac{l(\vecs^{\ast})}{8}-\frac{m}{16}\right), $$
      where $l(\vecs^{\ast})$ is the number of components in $\vecs^{\ast}$ which are equal to zero.
      Since we obviously have $l(\vecs^{\ast})=m-l(\vecs)$, we find $S_2+S_3=0$.
      So far we have
      $$ I_1=S_1-\frac{1}{72\cdot 16^m}(2^m-1)^2(32\cdot 2^m-25\cdot 4^m-8). $$
      But since
      \begin{align*}  S_1=&\sum_{a,b=1}^{2^m-1}\int_{\frac{a-1}{2^m}}^{\frac{a}{2^m}}\int_{\frac{b-1}{2^m}}^{\frac{b}{2^m}}\Delta_1\left(\frac{a}{2^m},\frac{b}{2^m}\right)\Delta_2\left(\frac{a}{2^m},\frac{b}{2^m}\right)\rd \alpha \rd \beta \\
      =&\frac{1}{2^{2m}}\sum_{\alpha,\beta\in\qqq}\Delta_1(\alpha,\beta)\Delta_2(\alpha,\beta),
      \end{align*}
     we also know the value of $S_1$ from Lemma~\ref{gemischt}. Putting all results together, we obtain the claimed formula in Theorem~\ref{Theo}. $\hfill \Box \vspace{0.5cm}$

\bigskip

\noindent {\bf Acknowledgments.} The author would like to thank Friedrich Pillichshammer for valuable suggestions to improve the presentation.

\noindent{\bf Author's Address:}

\noindent Ralph Kritzinger, Institut f\"{u}r Finanzmathematik und angewandte Zahlentheorie, Johannes Kepler Universit\"{a}t Linz, Altenbergerstra{\ss}e 69, A-4040 Linz, Austria. Email: ralph.kritzinger(at)jku.at


\begin{thebibliography}{10} 

\bibitem{bil} D. Bilyk, V.N. Temlyakov, and R. Yu, Fibonacci sets and symmetrization in discrepancy
theory, J. Complexity 28, No. 1 (2012) 18--36.

\bibitem{daven} H. Davenport, Note on irregularities of distribution, Mathematika 3 (1956) 131--135.

\bibitem{DP10} J. Dick, F. Pillichshammer, Digital nets and sequences. Discrepancy theory and quasi-Monte Carlo integration, Cambridge University Press, Cambridge, 2010.	

\bibitem{DT97} M. Drmota, R.F. Tichy, Sequences, discrepancies and applications, Lecture Notes in Mathematics 1651, Springer Verlag, Berlin, 1997.

\bibitem{FPPS09} H. Faure, F. Pillichshammer, G. Pirsic, and W. Ch. Schmid, $L_2$-discrepancy of generalized two-dimensional Hammersley point
  sets scrambled with arbitrary permutations, Acta Arith. 141 (2010) 395--418.

\bibitem{hala} G. Hal\'{a}sz, On Roth's method in the theory of irregularities of point distributions, in: Recent progress in analytic number theory, Vol. 2, Academic Press, London-New York, 1981, pp. 79--94. 

\bibitem{Hor} J. Horbowicz, An asymptotic relation between the extreme discrepancy and the $L_p$-discrepancy, Monatsh. Math 90, No. 4 (1980) 297--301.

\bibitem{HKP14} A. Hinrichs, R. Kritzinger, and F. Pillichshammer, Optimal order of $L_p$ discrepancy of digit shifted Hammersley point sets in dimension 2,  Unif. Distrib. Theory 10 (2015) 115--133.

\bibitem{Kri2}  P. Kritzer, F. Pillichshammer, An exact formula for the $L_2$ discrepancy of the shifted Hammersley point set, Unif. Distrib. Theory 1 (2006) 1--13.

\bibitem{Kri1}  P. Kritzer, F. Pillichshammer, Point sets with low $L_p$ discrepancy, Math. Slovaca 57 (2007) 11--32. 

\bibitem{kuinie} L. Kuipers, H. Niederreiter, Uniform distribution of sequences, John Wiley, New York, 1974.

\bibitem{lp} G. Larcher, F. Pillichshammer, Walsh series
analysis of the $L_2$ discrepancy of symmetrisized point sets,
Monatsh. Math. 132 (2001) 1--18.

\bibitem{Lar}  G. Larcher, F. Pillichshammer, Sums of distances to the nearest integer and the discrepancy of digital nets, Acta Arith. 106 (2003) 379--408.

\bibitem{LP14} G. Leobacher, F. Pillichshammer, Introduction to quasi-Monte Carlo integration and applications. Compact Textbooks in Mathematics, Birkh\"auser, 2014.

\bibitem{mat} J. Matou\v{s}ek, Geometric discrepancy. An illustrated guide. Algorithms and Combinatorics 18, Springer-Verlag, Berlin, 1999.

\bibitem{Nied92} H. Niederreiter, Random number generation and quasi-Monte Carlo methods. Number 63 in CBMS-NFS Series in Applied Mathematics, SIAM, Philadelphia, 1992.

\bibitem{Pill}  F. Pillichshammer, On the $L_p$ discrepancy of the Hammersley point set, Monath. Math. 136 (2002) 67--79.

\bibitem{Roth} K.F. Roth, On irregularities of distribution, Mathematika 1 (1954) 73--79.

\bibitem{schX} W.M. Schmidt, Irregularities of distribution. X, in: Number Theory and Algebra, Academic Press, New York, 1977, pp. 311--329.

\end{thebibliography}
\end{document}